\documentclass{amsart}
\usepackage[utf8]{inputenc}
\usepackage[english]{babel}
\usepackage[T1]{fontenc}

\usepackage{amsmath}
\usepackage{amsfonts}
\usepackage{amssymb}
\usepackage{amsaddr}
\usepackage{amsthm} 
\usepackage{hyperref} 
\usepackage{stmaryrd} 
\usepackage{yhmath} 
\usepackage{tikz}
\usepackage{subcaption}
\usepackage{accents} 
\usepackage{xcolor} 
\usepackage{dsfont} 

\usetikzlibrary{patterns}

\usepackage{enumitem} 

\usepackage[left=2cm,right=2cm,top=2cm,bottom=2cm]{geometry}
\usepackage{verbatim}

\DeclareMathOperator{\N}{\mathbb{N}}

\DeclareMathOperator{\R}{\mathbb{R}}

\DeclareMathOperator{\E}{\mathbb{E}}
\DeclareMathOperator{\1}{\mathds{1}}

\DeclareMathOperator{\supp}{\mathrm{supp}}

\DeclareFontFamily{U}{mathx}{}
\DeclareFontShape{U}{mathx}{m}{n}{<-> mathx10}{}
\DeclareSymbolFont{mathx}{U}{mathx}{m}{n}
\DeclareMathAccent{\widehat}{0}{mathx}{"70}
\DeclareMathAccent{\widecheck}{0}{mathx}{"71}

\usepackage[backend=biber,style=numeric,sortcites=true]{biblatex}
\usepackage{csquotes}
\numberwithin{equation}{section}

\makeatletter
\def\blfootnote{\xdef\@thefnmark{}\@footnotetext}
\makeatother

\theoremstyle{plain}
\newtheorem{theorem}{Theorem}[section]

\newtheorem{remark}[theorem]{Remark}
\newtheorem{lemma}[theorem]{Lemma}
\newtheorem{proposition}[theorem]{Proposition}
\newtheorem{corollary}[theorem]{Corollary}
\newtheorem{hypothesis}[theorem]{Hypothesis}

\title[Non-$L^1$-Asymptotic completeness for Vlasov-Maxwell]{A note on the non $L^1$-Asymptotic completeness of the Vlasov-Maxwell system}
\author{Emile Breton}
\address{Univ Rennes, CNRS, IRMAR - UMR 6625, F-35000 Rennes, France \\ Email address : emile.breton@univ-rennes.fr}

\begin{document}


\begin{abstract}
    We prove that under a generic asymptotic condition on the charge, the small data solutions to the Vlasov-Maxwell system do not verify linear scattering. In other words, we show the non-$L^1$ asymptotic completeness of the system. The proof makes use of the Lorentz invariance of the equations.
\end{abstract}



\blfootnote{This work was conducted within the the France 2030 program, Centre Henri Lebesgue ANR-11-LABX-0020-01}


\maketitle

\section{Introduction}
\subsection{General context}

The relativistic Vlasov-Maxwell system models a collisionless plasma, it can be written as 
\begin{equation}
    \tag{RVM}
    \label{equation_RVM1}
    \partial_t f_\alpha+\widehat{v_\alpha}\cdot\nabla_x f_\alpha +e_\alpha(E+\widehat{v_\alpha}\times B)\cdot\nabla_v f_\alpha=0,\qquad 1\leq \alpha \leq N,
\end{equation}
\begin{align*}
    \partial_tE=\nabla\times B -4\pi j,&\qquad \nabla\cdot E=4\pi\rho,\\
    \partial_t B=-\nabla\times E,&\qquad  \nabla\cdot B=0,
\end{align*}
where $\rho,j$ are the charge and current density of the plasma defined by 
\begin{equation*}
    \rho=\sum_{1\leq\alpha\leq N} e_\alpha\int_{\R^3_v} f_\alpha \mathrm{d}v,\quad j=\sum_{1\leq\alpha\leq N}e_\alpha\int_{\R^3_v}\widehat{v_\alpha}f_\alpha\mathrm{d}v.
\end{equation*}
We consider the multi-species case $N\geq 2$, where $f_\alpha(t,x,v)$ is the density function of a species $\alpha$ with mass $m_\alpha>0$ and charge $e_\alpha\neq 0$. Here, $t\geq 0$ denotes the time, $x\in\R^3_x$ the particle position and $v\in\R^3_v$ the particle momentum. Moreover, $(E,B)(t,x)$ denotes the electromagnetic field of the plasma. For a particle of mass $1$ and momentum $v\in\R^3_v$, we will denote its energy by $v^0:=\langle v\rangle=\sqrt{1+|v|^2}$ and its relativistic speed as 
    \begin{equation}
        \widehat{v}:=\frac{v}{v^0},\quad v\in\R_v^3.
    \end{equation}
We will also consider the inverse of $v\mapsto\widehat{v}$, defined as 
\begin{equation}
    u\in\{u\in\R^3_v\,|\,|u|<1\}\longmapsto \widecheck{u}:=\frac{u}{\sqrt{1-|u|^2}}.
\end{equation}
In addition, we denote $v_\alpha(v):=\frac{v}{m_\alpha}$. We will simply write $v_\alpha$ for $v_\alpha(v)$ since there is no risk of confusion. Then,  
    \begin{equation}
        \widehat{v_\alpha}=\frac{v}{\sqrt{m_\alpha^2+|v|^2}}.
    \end{equation}
For any $\alpha$, we will also write 
\begin{equation}
    v_\alpha^0:=\sqrt{m_\alpha^2+|v|^2},
\end{equation}
so that $\widehat{v_\alpha}=\frac{v}{v_\alpha^0}$.
Finally, the initial data $f_{\alpha0}=f_\alpha(0,\cdot)$ and $(E,B)(0,\cdot)=(E_0,B_0)$ also satisfy, in the electrically neutral setting, the constraint equations 
\begin{equation}
    \label{equation_condition_condition_initiale}
    \nabla\cdot E_0=4\pi\sum_{1\leq\alpha\leq N} e_\alpha\int_{\R^3_v}  f_{\alpha0} \mathrm{d}v,\qquad \nabla\cdot B_0=0,\qquad \sum_{1\leq\alpha\leq N} e_\alpha \int_{\R^3_x\times\R^3_v}f_{\alpha0} \mathrm{d}v\mathrm{d}x=0.
\end{equation}

In 3D the global existence problem for the classical solutions to \eqref{equation_RVM1} is still open, though various continuation criteria have been proved (see, for instance, \cite{Luk_Strain_14}).
In contrast, small data solutions are known to be global in time and to disperse at the linear rate \cite{Glassey_Schaeffer_88,Glassey_Strauss_1987,Schaeffer_04,Bigorgne_20,Wang_22}. Note that the electromagnetic field does not have to be small \cite{Wei_Yang_21,Bigorgne_2023}. The asymptotic behavior of small data solutions has been fully understood recently, the electromagnetic field $(E,B)$ enjoys a linear scattering dynamic, whereas the densities $f_\alpha$ exhibit a modified scattering dynamic \cite{Bigorgne_2023,Pankavich_BenArtzi_2024,bigorgne_2023_scatteringmap,breton2025}, as stated in Theorem \ref{theoreme_anciens_resultats} below.\\
The purpose of this article is to rigorously prove that, for \eqref{equation_RVM1}, linear scattering is a non-generic phenomenon. It occurs only for a subset of initial data of codimension $1$. This means that we cannot approximate small data solutions by solutions to the linearized system around the trivial solution.\\
For the Vlasov-Poisson system, the non-$L^1$ asymptotic completeness was proved in \cite{Choi_Ha_11} whereas the modified scattering dynamic was fully described later in \cite{Ionescu_Pausader_22}.

\subsection{Main results}

We begin by recalling the general assumptions and the main theorem of \cite{breton2025}.
\begin{hypothesis}
\label{hypothese_suffisante}
Suppose $(f_\alpha,E,B)$ is a $C^1$ solution to \eqref{equation_RVM1} with initial data $(f_{\alpha0},E_0,B_0)$ and satisfying the following properties.
    \begin{itemize}
        \item There exists $k>0$ such that $f_{\alpha0}$ are nonnegative $C^1$ functions with support in $\{(x,v)\in\R^3_x\times\R^3_v\,|\,|x|\leq k,\,|v|\leq k\}$. Moreover, $E_0$ and $B_0$ are $C^1$ with support in $\{x\in\R^3_x\,|\,|x|\leq k\}$ and satisfy the constraint \eqref{equation_condition_condition_initiale}.
        \item There exists $C_0>0$ such that, for all $(t,x)\in\R_+\times\R^3_x$,
        \begin{align}
            \label{equation_estimee_EB_Glassey}
            |(E,B)(t,x)|&\leq\frac{C_0}{(t+|x|+2k)(t-|x|+2k)},\\
            \label{equation_estimee_gradEB_Glassey}
            |\nabla_x(E,B)(t,x)|&\leq\frac{C_0\log(t+|x|+2k)}{(t+|x|+2k)(t-|x|+2k)^2}.
        \end{align}
    \end{itemize}
\end{hypothesis}
\begin{remark}
    In the following, we will write $a\lesssim b$ when there exists $C>0$ independent of $t$ and depending on $(x,v)$ only through $k$. However, $C$ will usually depend on $(m_\alpha,e_\alpha)_{1\leq \alpha\leq N}$ and the initial data.
\end{remark}
\begin{remark}
     Any solution, arising from compactly supported data, constructed in \cite{Glassey_Schaeffer_88,Glassey_Strauss_1987,Schaeffer_04,Bigorgne_20,Wang_22,Bigorgne_2023} satisfies Hypothesis \ref{hypothese_suffisante}. Consequently, Theorems \ref{theoreme_anciens_resultats}--\ref{theoreme_pas_de_scattering_lineaire} hold for them.
\end{remark}
We recall a series of useful results from \cite{breton2025}.
\begin{theorem}
    \label{theoreme_anciens_resultats}
    Let $(f_{\alpha},E,B)$ be a solution to \eqref{equation_RVM1} satisfying Hypothesis \ref{hypothese_suffisante}. 
    \begin{itemize}
        \item First, $f_\alpha$ vanishes for high velocities. There exists $\beta>0$ such that for any $1\leq \alpha\leq N$ and all $t\geq0$,
    \begin{equation}
        \label{equation_support_f_alpha}
        \supp f_\alpha \subset\{(x,v)\in\R^3_x\times\R^3_v\,|\, |x|\leq \widehat{\beta}_{max}t+k,\,|v|\leq \beta\},\qquad \qquad\widehat{\beta}_{max}:=\max_{1\leq\alpha\leq N} \frac{\beta}{\sqrt{m_\alpha^2+\beta^2}}.
    \end{equation}
    \item For any $1\leq \alpha\leq N$, there exists $Q_\infty^\alpha\in C^0_c(\R^3_v)$ such that for all $t\geq 0$ and $v\in\R^3_v$,
    \begin{equation}
        \label{equation_corollaire_estimee_Q_infini}
        \bigg|\int_{\R^3_x}f_\alpha(t,x,v)\mathrm{d}x - Q^\alpha_\infty\left( \frac{v}{m_\alpha} \right)\bigg|\lesssim \frac{\log^5(2+t)}{2+t}.
    \end{equation}
    Moreover, for all $|x|<t$,
    \begin{equation}
        \label{equation_estimee_Q_infini_v}
        \bigg|t^3\int_{\R^3_v} f_\alpha(t,x,v)\mathrm{d}v-m_\alpha^3\left\langle\widecheck{x/t}\right\rangle^5 Q^{\alpha}_\infty\left(\widecheck{x/t}\right)\bigg|\lesssim \frac{\log^6(2+t)}{2+t}.
    \end{equation}
    \item There exist $\E,\mathbb{B}\in C^0(\R^3_v)$ such that, for all $t\geq 0$ and $|v|\leq\beta_{max}:=\max_{1\leq\alpha\leq N} \frac{\beta}{m_\alpha}$, we have
    \begin{equation}
        \label{equation_estimee_E_bb}
        |t^2 E(t,t\widehat{v})-\E(v)|+|t^2 B(t,t\widehat{v})-\mathbb{B}(v)|\lesssim \frac{\log^6(2+t)}{2+t}.
    \end{equation}
    \item Every density $f_\alpha$ verifies modified scattering. More precisely, for any $\alpha$, there exists $f_{\alpha\infty}\in C_c^0(\R^3_x\times\R^3_v)$ such that, for all $t>0$ large enough and all $(x,v)\in\R^3_x\times\R^3_v$,
    \begin{equation*}
        \left|f_\alpha\left(tv^0_\alpha-\log(t)\frac{e_\alpha}{v^0_\alpha}\E(v_\alpha)\cdot\widehat{v_\alpha},x+tv-\log(t)\frac{e_\alpha}{v^0_\alpha}\big(\E(v_\alpha)+\widehat{v_\alpha}\times \mathbb{B}(v_\alpha)\big),v\right)-f_{\alpha\infty}(x,v)\right|\lesssim \frac{\log^6(2+t)}{2+t}.
    \end{equation*}
    \end{itemize}
\end{theorem}

In the following, we will consider 
\begin{equation}
    \label{equation_definition_Q_infini}
    Q_\infty:=\sum_{1\leq \alpha \leq N}e_\alpha m_\alpha^3 Q_\infty^\alpha.
\end{equation}
This function will govern the asymptotic behavior of the current and charge density and will give us a necessary and sufficient condition for linear scattering. We now state our main result.

\begin{theorem}
    Let $(f_\alpha,E,B)$ be a solution to \eqref{equation_RVM1} that satisfies Hypothesis \ref{hypothese_suffisante}. Assume $Q_\infty\neq 0$. Then, there exists $\alpha\in\{1,\dots,N\}$ such that $f_\alpha$ does not satisfy a linear scattering dynamic. That is, for this specific $\alpha$, there exists no function $g_{\alpha\infty}\in L^1(\R^3_x\times\R^3_v)$ such that 
    \begin{equation}
        f_\alpha(tv^0_\alpha,x+tv,v)\xrightarrow[t\rightarrow +\infty]{L^1(\R^3_x\times\R^3_v)} g_{\alpha\infty}(x,v).
        \label{equation_fausse_linear_scattering1}
    \end{equation}
    \label{theoreme_pas_de_scattering_lineaire}
\end{theorem}
\begin{remark}
    For the linear scattering statement \eqref{equation_fausse_linear_scattering1} to be Lorentz invariant, it is important to parameterize the characteristics as $t\mapsto (tv^0_\alpha,x+tv,v)$ instead of $t\mapsto(t,x+t\widehat{v_\alpha},v)$.
\end{remark}
\begin{remark}
    Our result can be rewritten as an equivalence. In fact, if $Q_\infty=0$ then $\E=\mathbb{B}=0$ and thus from Theorem \ref{theoreme_anciens_resultats} we derive linear scattering. Hence, linear scattering holds if and only if $Q_\infty=0$.
\end{remark}
\begin{remark}
    As stated above, linear scattering can occur only when $Q_\infty=0$. In particular, this means $Q_\infty^\gamma\in\mathrm{Span}\{Q_\infty^\alpha\,|\, 1\leq\alpha\leq N,\,\alpha\neq\gamma\}$, so that the scattering data lie in a codimension 1 submanifold. Using the wave operator constructed in \cite{bigorgne_2023_scatteringmap}, we can link $f_{\alpha0}$ with $Q_\infty^\alpha$ and thus prove that the same holds for the initial data.
\end{remark}

\subsection{Ideas of the proof}

Let $(f_\alpha,E,B)$ be a solution to \eqref{equation_RVM1} that satisfies Hypothesis \ref{hypothese_suffisante} and recall the definition of $Q_\infty$ in \eqref{equation_definition_Q_infini}. Furthermore, we assume that $Q_\infty\neq0$ and the densities exhibit a linear scattering dynamic, i.e. equation \eqref{equation_fausse_linear_scattering1} holds for any $\alpha$.\\
We want to find a contradiction. We will first observe that, under the linear scattering hypothesis, for all $v\in\R^3_v$ such that $\mathbb{L}(v):=\mathbb{E}(v)+\widehat{v}\times\mathbb{B}(v)\neq 0$, 
\begin{equation}
    Q_\infty(v)=0.
\end{equation}
Hence, to prove the theorem, we want to negate the linear scattering property by finding $v\in\R^3_v$ such that $Q_\infty(v)\neq 0$ and $\mathbb{L}(v)\neq 0$. We already know (see for instance \cite[Proposition 3.5]{breton2025}) that $\mathbb{L}$ can be expressed in terms of $Q_\infty$. This expression is rather complex, and it is unclear whether we can obtain $\mathbb{L}(v)\neq0$ with assumptions on $Q_\infty(v)$ or not. However, if we assume $Q_\infty(0)\neq0$, we will be able to prove that there exists $v\in\R^3_v$ such that
\begin{equation*}
    Q_\infty(v)\neq 0,\qquad \mathbb{L}(v)\neq 0.
\end{equation*}
We obtain this result by exploiting Gauss's law and the asymptotic behavior of $E$ given by \eqref{equation_estimee_E_bb}.
In general, we do not have $Q_\infty(0)\neq0$. However, as long as $Q_\infty\neq 0$, this is always true in a certain inertial frame. More precisely, there exists $v\in\R^3_v$ and a Lorentz transform $A\in\mathrm{SO}_0(3,1)$ such that 
\begin{equation*}
    Q_\infty(v)\neq 0,\qquad A(v^0,v)=(1,0,0,0).
\end{equation*}
We conclude the proof by exploiting the Lorentz invariance of \eqref{equation_RVM1} and by showing that Hypothesis \ref{hypothese_suffisante} and the linear scattering hypothesis \eqref{equation_fausse_linear_scattering1} are preserved by composing with Lorentz transforms.

\section{General results}
In all the following sections, we consider $(f_\alpha,E,B)$ a solution to \eqref{equation_RVM1} that satisfies Hypothesis \ref{hypothese_suffisante}. 
Here, we show that, when $Q_\infty(0)\neq 0$, there exists $v\in\R^3_v$ such that $\mathbb{L}(v):=\mathbb{E}(v)+\widehat{v}\times\mathbb{B}(v)\neq0$ and $Q_\infty(v)\neq 0$. To do so, we first introduce a lemma that provides another relation between $Q_\infty$ and $\mathbb{E}$, where $Q_\infty$ is defined in \eqref{equation_definition_Q_infini}.

\begin{lemma}
    Let $0<\delta\leq \widehat{\beta}_{max}=\max_{1\leq\alpha\leq N} \frac{\beta}{\sqrt{m_\alpha^2+\beta^2}}$. The following equality holds
    \begin{equation}
        4\pi\int_{B(0,\delta)}\langle\widecheck{x}\rangle^5Q_\infty\left(\widecheck{x}\right)\mathrm{d}x=\int_{\mathbb{S}^2_\delta} \E\left(\widecheck{\omega}\right)\cdot\frac{\omega}{|\omega|}\mathrm{d}\mu_{\mathbb{S}^2_\delta}(\omega).
    \end{equation}
    \label{lemme_lien_entre_Q_et_E_bb}
\end{lemma}
\begin{proof}
    Let $\delta>0$. First, by the divergence theorem, we have
    \begin{equation*}
        \int_{B(0,t\delta)} \nabla\cdot E(t,x)\mathrm{d}x=\int_{\mathbb{S}^2_{t\delta}}E(t,\omega)\cdot \frac{\omega}{|\omega|}\mathrm{d}\mu_{\mathbb{S}^2_{t\delta}}=\int_{\mathbb{S}^2_{\delta}}t^2E(t,t\omega)\cdot \frac{\omega}{|\omega|}\mathrm{d}\mu_{\mathbb{S}^2_{\delta}}.
    \end{equation*}
    Now, recall \eqref{equation_estimee_E_bb} so that for any $\omega\in\mathbb{S}^2_{\delta}$, since $|\widecheck{\omega}|\leq\beta_{max}$,
    \begin{equation*}
        t^2E(t,t\omega)=\mathbb{E}(\widecheck{\omega})+ O\left(\frac{\log^6(2+t)}{2+t}\right).
    \end{equation*}
    Consequently,
    \begin{equation}
        \label{equation_egalite_div_E_E_bb}
         \int_{B(0,t\delta)} \nabla\cdot E(t,x)\mathrm{d}x=\int_{\mathbb{S}^2_{\delta}}\mathbb{E}(\widecheck{\omega})\cdot \frac{\omega}{|\omega|}\mathrm{d}\mu_{\mathbb{S}^2_{\delta}}+ O\left(\frac{\log^6(2+t)}{2+t}\right).
    \end{equation}
    However, by Gauss's law in \eqref{equation_RVM1}, we know that 
    \begin{equation}
        \label{equation_egalite_div_E_rho_preuve1}
         \int_{B(0,t\delta)} \nabla\cdot E(t,x)\mathrm{d}x=4\pi\int_{B(0,t\delta)}\rho(t,x)\mathrm{d}x=4\pi\int_{B(0,t\delta)}\sum_{1\leq\alpha\leq N} e_\alpha\int_{\R^3_v} f_\alpha(t,x,v) \mathrm{d}v\mathrm{d}x.
    \end{equation}
    Then, by \eqref{equation_estimee_Q_infini_v}, we have for $1\leq\alpha\leq N$
    \begin{equation*}
        \int_{\R^3_v} f_\alpha(t,x,v) \mathrm{d}v = \frac{1}{t^3}m_\alpha^3\left\langle\widecheck{x/t}\right\rangle^5 Q^{\alpha}_\infty\left(\widecheck{x/t}\right)+O\left(\frac{\log^6(2+t)}{(2+t)^4}\right).
    \end{equation*}
    This implies, after the change of variables $y\mapsto x/t$, 
    \begin{equation}
    \label{equation_equality_integral_Q_infty}
         4\pi\int_{B(0,t\delta)}\sum_{1\leq\alpha\leq N} e_\alpha\int_{\R^3_v} f_\alpha(t,x,v) \mathrm{d}v\mathrm{d}x= 4\pi\int_{B(0,\delta)}\langle\widecheck{x}\rangle^5 Q_\infty\left(\widecheck{x}\right)\mathrm{d}x+O\left(\frac{\log^6(2+t)}{2+t}\right).
    \end{equation}
    The result follows from \eqref{equation_egalite_div_E_E_bb},  \eqref{equation_egalite_div_E_rho_preuve1}, and \eqref{equation_egalite_div_E_E_bb}.
\end{proof}

\begin{proposition}    
    \label{corollaire_Q_infini_L_non_nul}
   Assume $Q_\infty(0)\neq 0$. Then, there exists $v\in\R^3_v$ such that,
    \begin{equation*}
        Q_\infty(v)\neq 0,\quad \mathbb{L}(v)\neq 0.
    \end{equation*}
\end{proposition}
\begin{proof}
    Since $Q_\infty$ is continuous, there exists $0<\delta\leq\widehat{\beta}_{max}$ such that for all $v\in \bar{B}(0,\delta)$, up to considering $-Q_\infty$,
    \begin{equation*}
        Q_\infty(\widecheck{v})>0 .
    \end{equation*}
    Using Lemma \ref{lemme_lien_entre_Q_et_E_bb}, this implies
    \begin{equation*}
        \int_{\mathbb{S}^2_\delta} \E\left(\widecheck{\omega}\right)\cdot\frac{\omega}{|\omega|}\mathrm{d}\mu_{\mathbb{S}^2_\delta}(\omega)>0.
    \end{equation*}
    It follows that there exists $v\in \R^3_v$ such that 
    \begin{equation*}
        Q_\infty(v)\neq 0,\quad \mathbb{E}(v)\cdot\widehat{v}\neq 0.
    \end{equation*}
    The result follows from the expression of $\mathbb{L}(v)=\E(v)+\widehat{v}\times\mathbb{B}(v)$, so that $\mathbb{L}(v)\cdot \widehat{v}=\mathbb{E}(v)\cdot \widehat{v}$.
\end{proof}


In the remainder of the article, we assume that the densities $f_\alpha$ satisfy the linear scattering assumption, that is, \eqref{equation_fausse_linear_scattering1} holds for any $\alpha$.

\begin{proposition}
    \label{proposition_Q_infini_nul_quand_L_non_nul}
    Let $v\in\R^3_v$ such that $\mathbb{L}(v)\neq 0$. Then, for any $1\leq\alpha\leq N$,
    \begin{equation}
        Q_\infty^\alpha(v)=0.
    \end{equation}
    In particular,
    \begin{equation}
        Q_\infty(v)=0.
    \end{equation}
    
\end{proposition}
\begin{proof}
    We begin by considering the densities $f_\alpha$ composed with the linear flow
    \begin{equation*}
        g_\alpha(t,x,v):=f_\alpha(t,x+t\widehat{v_\alpha},v).
    \end{equation*}
    Since $L^1$ scattering holds, we know that there exists a sequence $(t_k)_{k\in\N}$ such that, $t_k\xrightarrow[k\rightarrow +\infty]{}+\infty$ and 
    \begin{align*}
        g_\alpha(t_kv_\alpha^0,x,v)\xrightarrow[k\rightarrow +\infty]{}g_{\alpha\infty}(x,v),&\quad \textrm{ for almost all } (x,v)\in\R^3_x\times\R^3_v,\\
        \int_{\R^3_x}g_\alpha(t_kv^0_\alpha,x,v)\mathrm{d}x\xrightarrow[k\rightarrow +\infty]{}\int_{\R^3_x}g_{\alpha\infty}(x,v)\mathrm{d}x,&\quad \textrm{ for almost all } v\in\R^3_v.
    \end{align*}
    Using the convergence of $\int_{\R^3_x} g_\alpha(t,x,v)\mathrm{d}x=\int_{\R^3_x} f_\alpha(t,x,v)\mathrm{d}x$ in \eqref{equation_corollaire_estimee_Q_infini}, we have 
    \begin{equation*}
        Q_\infty^{\alpha}(v_\alpha)=\int_{\R^3_x}g_{\alpha\infty}(x,v)\mathrm{d}x,\quad \textrm{ for almost all } v\in\R^3_v.
    \end{equation*}
    As defined in \cite{breton2025}, let 
    \begin{equation*}
        h_\alpha(t,x,v):=g_\alpha\left(t,x+\frac{e_\alpha}{v^0_\alpha}\log(t)\Big[\widehat{v_\alpha}(\widehat{v_\alpha}\cdot\mathbb{L}(v_\alpha))-\mathbb{L}(v_\alpha)\Big],v\right).
    \end{equation*}
    By \cite[Corollary 3.12]{breton2025} we know that $h_\alpha$ is compactly supported. Moreover, 
    \begin{equation*}
        g_\alpha(t,x,m_\alpha v)=h_\alpha\left(t,x-\log(t)\frac{e_\alpha}{m_\alpha v^0}\Big[\widehat{v}(\widehat{v}\cdot\mathbb{L}(v))-\mathbb{L}(v)\Big],m_\alpha v\right),
    \end{equation*} 
    and then $g_\alpha(t,x,m_\alpha v)\xrightarrow[t\rightarrow +\infty]{}~0$ when $\left[\widehat{v}(\widehat{v}\cdot\mathbb{L}(v))-\mathbb{L}(v)\right]\neq 0$. Finally, since $\left[\widehat{v}(\widehat{v}\cdot\mathbb{L}(v))-\mathbb{L}(v)\right]\neq 0$ whenever $\mathbb{L}(v)\neq 0$, we derive 
    \begin{equation*}
        g_\alpha(t,x,m_\alpha v)\xrightarrow[t\rightarrow +\infty]{}0.
    \end{equation*}
    Hence, for almost all $(x,v)\in\R^3_x\times\R^3_v$ such that $ \mathbb{L}(v)\neq 0$, 
    \begin{equation*}
        g_{\alpha\infty}(x,m_\alpha v)=0.
    \end{equation*}
    Note that here, by continuity of $\mathbb{L}$, the equality holds at least on a small ball, which is of positive measure. Hence, for almost all $v\in\R^3_v$ such that $ \mathbb{L}(v)\neq 0$,
    \begin{equation*}
        Q_\infty^\alpha(v)=0.
    \end{equation*}
    Since both functions are continuous, the equality holds for all $v$ verifying the property. The result follows.
\end{proof}

\section{The restricted Lorentz group}

We begin by recalling that we consider $(f_\alpha,E,B)$ a solution to \eqref{equation_RVM1} that satisfies Hypothesis \ref{hypothese_suffisante}. Moreover, we assume that the densities $f_\alpha$ satisfy the linear scattering hypothesis \eqref{equation_fausse_linear_scattering1} and that $Q_\infty\neq 0$. We recall that the final goal is to reach a contradiction. In order to do so, we begin by introducing the restricted Lorentz group and giving general properties. 
\label{section_lorentz_boosts}

\subsection{General properties of the restricted Lorentz group}

We consider the restricted Lorentz group $\mathrm{SO}_0(3,1)$. It contains Lorentz transformations that preserve the orientation of space and the direction of time. We recall that Lorentz transformations are linear isometries of the Minkowski spacetime $(\R^{1+3},\eta)$, where $\eta:=\mathrm{diag}(-1,1,1,1)$ is the Minkowski metric. The restricted Lorentz group is generated by rotations, defined as 
\begin{equation*}
    R:=\left(\begin{array}{c|ccc}
        1& 0 & 0 & 0  \\
        \hline
        0 &  &  & \\
     0&&\widetilde{R}&\\
      0& & &
    \end{array}\right),
\end{equation*}
where $\widetilde{R}\in \mathrm{SO_3(\R)}$ is a 3D rotation, and Lorentz boosts $A_\varphi$ on the first spatial coordinate, defined as 
\begin{equation*}
    A_{\varphi}:=\begin{pmatrix}
        \cosh(\varphi)& \sinh(\varphi) & 0 & 0  \\
     \sinh(\varphi) & \cosh(\varphi) & 0 & 0 \\
     0&0&1&0 \\
     0&0&0&1
    \end{pmatrix},
\end{equation*}
where $\varphi\in\R$. We note that $A_\varphi^{-1}=A_{-\varphi}$. More precisely, the following proposition holds.
\begin{proposition}
    \label{proposition_ecriture_boosts}
    Let $A\in\mathrm{SO}_0(3,1)$. There exist $\varphi\in\R$ and rotations $R_1,R_2\in\mathrm{SO_0(3,1)}$ such that
    \begin{equation}
        A=R_1 A_{\varphi}R_2.
    \end{equation}
\end{proposition}
\begin{remark}
    In the following, since we can derive any Lorentz transformation of~~$\mathrm{SO_0(3,1)}$ by multiplying rotations and $A_\varphi$, we will limit ourselves to the study of $A_\varphi$ and rotations $R$.
\end{remark}
Let us give some properties of the restricted Lorentz group.
\begin{lemma}
    \label{lemme_metrique_minkowski}
    Recall the Minkowski metric $\eta=\mathrm{diag}(-1,1,1,1)$. For all $A\in\mathrm{SO_0(3,1)}$ and $X=(t,x)\in\R^{1+3}$, we have $\eta(AX,AX)=\eta(X,X)$. This rewrites as 
    \begin{equation*}
        |A^s(t,x)|^2-|A^0(t,x)|^2=|x|^2-|t|^2,
    \end{equation*}
    where $A^s(t,x):=\left(A^1(t,x),A^2(t,x),A^3(t,x)\right)\in\R^3$ is the spatial component of $A(t,x)=(A^0(t,x),A^s(t,x))\in\R^4$. Moreover,
    \begin{equation*}
        |A^1_{\varphi}(t,x)|^2-|A^0_{\varphi}(t,x)|^2=|x^1|^2-|t|^2.
    \end{equation*}
\end{lemma}
We will use that $\mathrm{SO}_0(3,1)$ acts transitively on the connected components of the hyperboloids.
\begin{lemma}
    Let $(t,x)\in\R_+\times\R^3_x$ with $|x|<t$. Then, there exists $A\in\mathrm{SO}_0(3,1)$, depending on $(t,x)$, such that 
    \begin{equation*}
        A(t,x)=\left(\sqrt{t^2-|x|^2},0,0,0\right).
    \end{equation*}
    \label{lemme_existence_A_nul_en_espace}
\end{lemma}

\subsection{Composing the solutions with the restricted Lorentz group}
\label{sous_section_def_lorentz_boosts}

As stated in the introduction, we want to compose our solution $(f_\alpha,E,B)$ with elements of the restricted Lorentz group. In order to do so, it is convenient to introduce the Faraday tensor $F$ associated to the electromagnetic field $(E,B)$. It is the antisymmetric matrix defined as 
\begin{equation*}
    F:= \begin{pmatrix}
            0 & E^1 & E^2 & E^3 \\
         -E^1 & 0 & B^3 & -B^2 \\
         -E^2 & -B^3 & 0 & B^1 \\
         -E^3 & B^2 & -B^1 & 0
        \end{pmatrix}.
\end{equation*}
In the following, we will write $(f_\alpha,F)$ as a solution to \eqref{equation_RVM1} instead of $(f_\alpha,E,B)$. Moreover, for convenience, we begin by uniquely extending $(f_\alpha,F)$ by 0 to $\{(t,x)\in\R\times\R^3_x\,|\, t\leq 0,\, |x|\geq -t +k\}\times\R^3_v$, with a slight abuse of notation since $F$ does not depend on $v\in\R^3_v$. Note that this is possible since this domain is globally hyperbolic. Then, $(f_\alpha,F)$ is now defined on $\Omega\times\R^3_v$, where 
\begin{equation*}
    \Omega:=\left(\R_+\times\R^3_x\right)\cup\{(t,x)\in\R\times\R^3_x\,|\, t\leq 0,\, |x|\geq -t +k\}.
\end{equation*}
Now, for $A\in\mathrm{SO}_0(3,1)$, we compose $(f_\alpha,F)$ with $A$ and define
\begin{equation}
    \label{equation_premiere_definition_f_alpha_varphi}
    f_\alpha^A(t,x,v):=f_\alpha\left(A(t,x),A^s(v^0_\alpha,v)\right).
\end{equation}
Moreover, we define the new Faraday tensor as the matrix product
\begin{equation}
    \label{equation_premiere_definition_F_varphi}
    F^A(t,x):=A^{-1}F(A(t,x))\left(A^{-1}\right)^T.
\end{equation}
This can be rewritten in index notation as  $(F^A)^{\mu\nu} = {(A^{-1})^{\mu}}_\alpha F^{\alpha\beta}(A\cdot) {(A^{-1})^\nu}_\beta$, where we raise and lower indices with respect to the Minkowski metric. We then denote the associated electromagnetic field by $(E^A,B^A)$, which is, roughly speaking, given by the cartesian components of $F^A$. By the Lorentz invariance of the Vlasov-Maxwell system, $(f^A,F^A)$ is a solution to \eqref{equation_RVM1} on $A^{-1}\Omega\times\R^3_v$. As we shall see, $[T_A,+\infty)\times\R^3_x\times\R^3_v\subset A^{-1}\Omega\times\R^3_v$ for $T_A>0$ large enough. This can be seen on Figure \ref{figure_domaine_apres_composition} below.
    \begin{figure}[h!]
    \centering
    \begin{subfigure}{0.4\textwidth}
       \begin{tikzpicture}[thick,scale=1, every node/.style={scale=1}]
            \draw[red,dashed, very thick] (-3,2)--(3,2);
            \draw (3,2) node [above]{$t=T_\varphi$};
            \draw[->] (-3.1,0) -- (3.1,0);
            \draw (3.1,0) node[right] {$x$};
            \draw [->] (0,-1) -- (0,3.1);
            \draw (0,3.1) node[above] {$t$};
            
        \end{tikzpicture}
        \caption{Representation of $t=T_\varphi$ in the inertial frame $(t,x)$ \\ \\}
        \label{figure_domaine}
    \end{subfigure}
    \begin{subfigure}{0.4\textwidth}
    \begin{tikzpicture}[thick,scale=1, every node/.style={scale=1}]
            \fill[fill=gray!10]
            plot[domain=1:2](0.5*\x,{-abs(\x)+1})
            --(2,-1)--(3,-1)
            -- (3,0)
            -- cycle;
            \fill[fill=gray!10]
            plot[domain=-1:-2](0.5*\x,{-abs(\x)+1})
            --(-2,-1)--(-3,-1)
            -- (-3,0)
            -- cycle;
            \fill[fill=gray!10]
            (-3,0)--(-3,3)--(3,3)--(3,0)
            -- cycle;
            \fill[fill=gray!40]
            plot[domain=-4:-1](0.5*\x, {abs(\x)-1)})
            --plot[domain=1:4](0.5*\x, {abs(\x)-1)})
            --cycle ;
            \draw [dashed]plot[domain=0:3](0.5*\x,{abs(\x)});
            \draw (1.5,3) node[above]{$t'=|x'|$};
            \draw [dashed]plot[domain=0:-3](0.5*\x,{abs(\x)});
            \draw [black, thick] plot[domain=-4:-1](0.5*\x, {abs(\x)-1)});
            \draw [black, thick] plot[domain=1:4](0.5*\x, {(abs(\x)-1)});
            \draw [black, thick] plot[domain=1:2](0.5*\x,{-abs(\x)+1});
            \draw [black, thick] plot[domain=-1:-2](0.5*\x,{-abs(\x)+1});
            \draw[dashed, red, very thick] plot[domain=-2.25:3.125](0.5*\x, {4/3-(sqrt(5)/3)*\x});
            \draw (1.5,-0.5) node[right]{$t=T_\varphi$};
            
            \draw (-2.5,2.5) node {$A_\varphi^{-1}\Omega$};

            
            \draw[->] (-3.1,0) -- (3.1,0);
            \draw (3.1,0) node[right] {$x'$};
            \draw [->] (0,-1) -- (0,3.1);
            \draw (0,3.1) node[above] {$t'$};
            
        \end{tikzpicture}
        \caption{Representation of the domain $A_\varphi^{-1}\Omega=\{(t,x)\in\R\times\R^3_x\,|\, A_\varphi(t,x)\in \Omega\}$ (in light and dark gray) and $t=T_{\varphi}$ in the inertial frame $(t',x')=A_\varphi(t,x)$ }
        \end{subfigure}
        \caption{Representation of $t=T_\varphi$ in the inertial frames $(t,x)$ and $(t',x')=A_\varphi(t,x)$ with $\varphi<0$}
        \label{figure_domaine_apres_composition}
\end{figure}
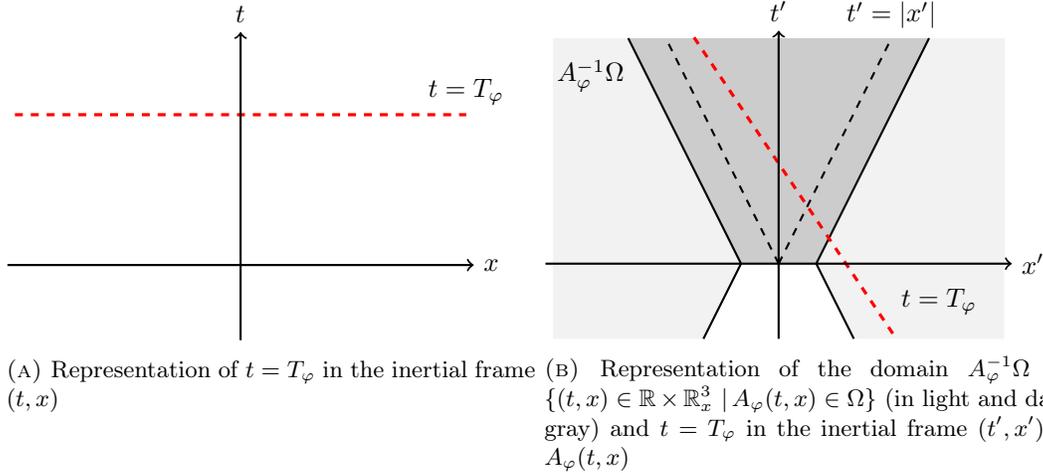

\subsection{The case of rotations}
We detail here what happens when we compose $(f_\alpha,F)$ with a rotation $R\in\mathrm{SO_0(3,1)}$. Let $R\in\mathrm{SO_0(3,1)}$ be a rotation and $(f^{R}_\alpha,F^{R})$ be defined as in \eqref{equation_premiere_definition_f_alpha_varphi}--\eqref{equation_premiere_definition_F_varphi}. We recall that $R$ is the block diagonal matrix composed of the blocks $1$ and $\widetilde{R}\in\mathrm{SO_3(\R)}$. Moreover, $R^0(t,x)=t$, and $\widetilde{R}$ preserves the Euclidean norm. Hence, $R^{-1}\Omega=\Omega$ and we can simply consider $T_R=0$. Equations \eqref{equation_premiere_definition_f_alpha_varphi}--\eqref{equation_premiere_definition_F_varphi} rewrite
\begin{equation*}
    f_\alpha^{R}(t,x,v)=f_\alpha\left(t,\widetilde{R}(x),\widetilde{R}(v)\right),\qquad F^R(t,x)=R^{-1}F(t,\widetilde{R}(x))\left(R^{-1}\right)^T.
\end{equation*}
In this case, it is straightforward to check that Hypothesis \ref{hypothese_suffisante} and the linear scattering hypothesis \eqref{equation_fausse_linear_scattering1} are verified.

\begin{proposition}
    \label{proposition_proprietes_conservees_rotations}
    $(f^{R}_\alpha,F^{R})$ satisfies Hypothesis \ref{hypothese_suffisante} and the linear scattering hypothesis \eqref{equation_fausse_linear_scattering1}. More precisely, for any $1\leq\alpha\leq N$, we have
    \begin{equation*}
        f^{R}_\alpha(tv^0_\alpha,x+tv,v)\xrightarrow[t\rightarrow +\infty]{L^1(\R^3_x\times\R^3_v)} g_{\alpha\infty}\big(\widetilde{R}(x),\widetilde{R}(v)\big).
    \end{equation*}
    Moreover, 
    \begin{equation*}
        \left|\int_{\R^3_x}f^{R}_\alpha(t,x,v)\mathrm{d}x - Q_\infty^\alpha\big(\widetilde{R}(v_\alpha)\big)\right|\lesssim \frac{\log^5(2+t)}{2+t}.
    \end{equation*}
\end{proposition}
\begin{proof}
    Since $\widetilde{R}$ preserves the norm and has Jacobian 1, Hypothesis \ref{hypothese_suffisante} is directly verified.\\
    By the change of variables $(y,w)=(\widetilde{R}(x),\widetilde{R}(v))$, since $\sqrt{m_\alpha^2+|\widetilde{R}(v)|^2}=\sqrt{m_\alpha^2+|v|^2}=v_\alpha^0$, we have
    \begin{align*}
        \int_{\R^3_x\times\R^3_v} \left|f^{R}_\alpha(tv^0_\alpha,x+tv,v)-g_{\alpha\infty}\big(\widetilde{R}(x),\widetilde{R}(v)\big)\right|\mathrm{d}v\mathrm{d}x&=\int_{\R^3_x\times\R^3_v} \left|f_\alpha\left(tv^0_\alpha,\widetilde{R}(x+tv),\widetilde{R}(v)\right)-g_{\alpha\infty}\big(\widetilde{R}(x),\widetilde{R}(v)\big)\right|\mathrm{d}v\mathrm{d}x,\\
        &=\int_{\R^3_x\times\R^3_v} \left|f_\alpha(tv^0_\alpha,x+tv,v)-g_{\alpha\infty}(x,v)\right|\mathrm{d}v\mathrm{d}x.
    \end{align*}
    The convergence follows by the linear scattering hypothesis. Finally, the convergence of $\int_{\R^3_x}f_\alpha^R(t,x,v)\mathrm{d}x$ is a direct consequence of \eqref{equation_corollaire_estimee_Q_infini}.
\end{proof}


\subsection{First properties of the Lorentz boost}
In the following two subsections, we focus on the functions arising from the composition of $(f_\alpha,F)$ with the boost $A_\varphi$. As in \eqref{equation_premiere_definition_f_alpha_varphi}--\eqref{equation_premiere_definition_F_varphi}, we define
\begin{equation}
    \label{equation_deuxieme_definition_f_et_F_varphi}
    f_\alpha^\varphi(t,x,v):= f_\alpha\left(A_\varphi(t,x),A^s_\varphi(v^0_\alpha,v)\right),\qquad F^\varphi(t,x):=A_{-\varphi}F(A_\varphi(t,x))A_{-\varphi}.
\end{equation}
As stated above, these functions are defined on $A_\varphi^{-1}\Omega\times\R^3_v$. We show that for $T_\varphi$ large enough, 
$[T_\varphi,+\infty)\times\R^3_x\times\R^3_v\subset A_\varphi^{-1}\Omega\times\R^3_v$.
\begin{proposition}
    \label{proposition_hypothese_stable_par_lorentz_boost}
    There exists $T_\varphi>0$ large enough such that $(f_\alpha^\varphi,F^\varphi)$ is a solution to \eqref{equation_RVM1} on $[T_\varphi,+\infty)\times\R^3_x\times\R^3_v$. Moreover, $(f_\alpha^\varphi,F^\varphi)$ satisfies Hypothesis \ref{hypothese_suffisante}.
\end{proposition}
\begin{proof}
    \textbf{Step 1:} We first need to prove that $[T_\varphi,+\infty)\times\R^3_x\times\R^3_v\subset A^{-1}_\varphi\Omega$. As can be seen in Figure \ref{figure_domaine_apres_composition}, it suffices to prove that, for $T_\varphi>0$ large enough, the set $\{t=T_\varphi\}$ does not intersect with $\widetilde{\Omega}:=\{(t,x)\in\R\times\R^3_x\,|\, t'\leq 0,\, |x'|< -t' +k\}$, represented in white in Figure \ref{figure_domaine_apres_composition}. Indeed, since 
    \begin{equation*}
        \{t=T_\varphi\}=\left\{(t,x)\in\R\times\R^3_x\,\Big|\, t'=\frac{T_\varphi}{\cosh(\varphi)}+\frac{\sinh(\varphi)}{\cosh(\varphi)}(x')^1\right\},
    \end{equation*}
    and $\frac{|\sinh(\varphi)|}{\cosh(\varphi)}<1$, $\{t=T_\varphi\}$ will either intersect with $\widetilde{\Omega}$ or always remain above it. In particular, it is enough to prove that, for $(x')^1=\pm k$, we have $t'=\frac{T_\varphi}{\cosh(\varphi)}+\frac{\sinh(\varphi)}{\cosh(\varphi)}(x')^1\geq 0$. We can then consider $T_\varphi:=|\sinh(\varphi)|k$ and derive the result.
        
    \textbf{Step 2: }We show that the initial data $(f_\alpha^\varphi,F^\varphi)(T_\varphi,\cdot)$ are compactly supported. By finite speed of propagation, for all $t\geq 0$ and $v\in\R^3_v$, we have
    \begin{equation}
        \label{equation_support_f_et_F}
        \supp \big(f_\alpha(t,\cdot,v),F(t,\cdot)\big)\subset \{x\in\R^3_x\,|\, |x|\leq t+k\}.
    \end{equation}
    Hence, $f_\alpha^\varphi(T_\varphi,\cdot,v)$ is supported where $|A_\varphi^s(T_\varphi,x)|-A_\varphi^0(T_\varphi,x)\leq k$. Moreover, since $|A_\varphi^s(T_\varphi,x)|+A_\varphi^0(t,x)\leq (\cosh(\varphi)+|\sinh(\varphi)|)(T_\varphi+|x|)$,  Lemma \ref{lemme_metrique_minkowski} yields
    \begin{equation*}
    	|x|-T_\varphi=(|A_\varphi^s(T_\varphi,x)|-A_\varphi^0(T_\varphi,x))\frac{|A_\varphi^s(T_\varphi,x)|+A_\varphi^0(T_\varphi,x)}{T_\varphi+|x|}\leq (|A_\varphi^s(T_\varphi,x)|-A_\varphi^0(T_\varphi,x))(\cosh(\varphi)+|\sinh(\varphi)|).
    \end{equation*}
    Consequently,
    \begin{equation*}
        \supp(f_\alpha^\varphi(T_\varphi,\cdot,v),F^\varphi(T_\varphi,\cdot))\subset \left\{x\in\R^3_x\,|\, |x|\leq T_\varphi +k\big(\cosh(\varphi)+|\sinh(\varphi)|\big)\right\}.
    \end{equation*}
    Moreover, by \eqref{equation_support_f_alpha}, $ f_\alpha^\varphi(T_\varphi,x,\cdot)$ is supported where $A_\varphi^0(v^0_\alpha,v)\leq \sqrt{m_\alpha^2+\beta^2}$. Since $v^0_\alpha\leq A_\varphi^0(v^0_\alpha,v)(\cosh(\varphi)+|\sinh(\varphi)|)$, this implies
    \begin{equation*}
        \supp f_\alpha^\varphi(T_\varphi,x,\cdot)\subset \left\{v\in\R^3_v\,|\, |v|\leq \sqrt{m_\alpha^2+\beta^2}\big(\cosh(\varphi)+|\sinh(\varphi)|\big)\right\}.
    \end{equation*}
    We then derive the compact support hypothesis by taking 
    \begin{equation*}
        k_\varphi:=\max\left(T_\varphi +k(\cosh(\varphi)+|\sinh(\varphi)|),\sqrt{1+\beta^2}(\cosh(\varphi)+|\sinh(\varphi)|\right).
    \end{equation*}

    \textbf{Step 3:} Finally, it remains to prove the estimates on the fields. Recall that $A^0_\varphi(t,x)-|A^s_\varphi(t,x)|+k\geq 0$ as well as $t-|x|+k_\varphi\geq 0$ on the support of $F^\varphi$. In view of the estimates \eqref{equation_estimee_EB_Glassey}--\eqref{equation_estimee_gradEB_Glassey}, it suffices to prove
    \begin{equation}
        \label{equation_a_obtenir_dans_preuve_estimee_A_phi}
        A^0_\varphi(t,x)-|A^s_\varphi(t,x)|+2k\gtrsim t-|x|+2k_\varphi,\qquad A^0_\varphi(t,x)+|A^s_\varphi(t,x)|+2k\gtrsim t+|x|+2k_\varphi.
    \end{equation}
    For this, we mostly use the fact that for $a,b\geq 0$ and all $C>0$, we have $a+b\geq \min(1,1/C)(a+Cb)$. Then, since 
    \begin{equation*}
        \cosh(\varphi)+|\sinh(\varphi)|\geq\frac{A^0_\varphi(t,x)+|A^s_\varphi(t,x)|}{t+|x|}\geq \frac{1}{\sqrt{2}}(\cosh(\varphi)-|\sinh(\varphi)|),
    \end{equation*}
    we derive, with Lemma \ref{lemme_metrique_minkowski},
    \begin{align*}
        A^0_\varphi(t,x)-|A^s_\varphi(t,x)|+2k&\geq C\left(A^0_\varphi(t,x)-|A^s_\varphi(t,x)|+k+\frac{2\sqrt{2}k_\varphi}{\cosh(\varphi)-|\sinh(\varphi)|}\right)\\
        & \geq C\left(t-|x|+\left[\frac{2\sqrt{2}k_\varphi}{\cosh(\varphi)-|\sinh(\varphi)|}+k\right]\frac{A^0_\varphi(t,x)+|A^s_\varphi(t,x)|}{t+|x|}\right)\frac{t+|x|}{A^0_\varphi(t,x)+|A^s_\varphi(t,x)|}\\
       & \geq C\left(t-|x|+2k_\varphi+k\frac{1}{\sqrt{2}}(\cosh(\varphi)-|\sinh(\varphi)|)\right)\frac{1}{\cosh(\varphi)+|\sinh(\varphi)|}\\
        &\geq \widetilde{C}(t-|x|+2k_\varphi).
    \end{align*}
    We can then obtain \eqref{equation_a_obtenir_dans_preuve_estimee_A_phi} from the last two inequalities.
\end{proof}
\begin{remark}
    Here we derived functions which are defined only for $t\geq T_\varphi$. However, composing them with the translation in time $t\mapsto t+T_\varphi$, we derive a solution to \eqref{equation_RVM1} that satisfies Hypothesis \ref{hypothese_suffisante} for $t\geq 0$.
\end{remark}
\begin{corollary}
    Let $A\in\mathrm{SO_0(3,1)}$ and $(f^A_\alpha,F^A)$ be defined as in \eqref{equation_premiere_definition_f_alpha_varphi}--\eqref{equation_premiere_definition_F_varphi}. There exists $T_A>0$ large enough such that $(f^A_\alpha,F^A)$ is a solution to \eqref{equation_RVM1} on $[T_A,+\infty)\times\R^3_x\times\R^3_v$. Moreover, $(f^A_\alpha,F^A)$ satisfies Hypothesis \ref{hypothese_suffisante}.
    \label{corollaire_hypothese_verifiee_f_A}
\end{corollary}
\begin{proof}
    This follows directly from Propositions \ref{proposition_ecriture_boosts}, \ref{proposition_proprietes_conservees_rotations}, and \ref{proposition_hypothese_stable_par_lorentz_boost}.
\end{proof}
\begin{remark}
    Since $(f_\alpha^\varphi,F^\varphi)$ verifies Hypothesis \ref{hypothese_suffisante}, \cite[Propositions 2.5--2.6]{breton2025} hold. In particular, the following properties are verified by $g_\alpha^\varphi(t,x,v):=f_\alpha^\varphi(t,x+t\widehat{v_\alpha},v)$.
    There exists a constant $C>0$ such that, for all $t\geq T_\varphi$ and any $\alpha$,
        \begin{equation}
            \label{equation_support_g_alpha_varphi}
            \supp(g_\alpha^\varphi(t,\cdot))\subset \{(x,v)\in\R^3_x\times\R^3_v\,|\, |x|\leq C\log(2+t),\, |v|\leq \beta\}.
        \end{equation}
    Moreover, this implies that for any $1\leq\alpha\leq N$, and all $(t,x,v)$ with $t\geq T_\varphi$, we have
    \begin{align}
        |\partial_tg_\alpha^\varphi(t,x,v)|&\lesssim \frac{1}{2+t} \label{equation_estimee_dt_g}
    \end{align}
    \label{remarque_consequence_hypothese1}
\end{remark}

\subsection{Lorentz boosts and linear scattering}

We begin by showing that the densities $f^\varphi_\alpha$ satisfy the linear scattering assumption \eqref{equation_fausse_linear_scattering1}. This will prove that composing with an element of $\mathrm{SO_0(3,1)}$ preserves the linear scattering assumption. We then use this result to express the limit of $\int_{\R^3_x} f_\alpha^\varphi(t,x,v)\mathrm{d}x$ in terms of $Q_\infty^\alpha$ and $A_\varphi$.

\begin{proposition}
    The densities $f_\alpha^\varphi$ satisfy the linear scattering assumption, i.e. there exists $g_{\alpha\infty}^{\varphi}\in L^1(\R^3_x\times\R^3_v)$, such that
    \begin{equation}
        f_\alpha^\varphi(tv^0_\alpha,x+tv,v)\xrightarrow[t\rightarrow+\infty]{L^1_{x,v}}g^\varphi_{\alpha\infty}(x,v).
    \end{equation}
    Moreover, we have 
    \begin{equation}
        g^\varphi_{\alpha\infty}(x,v)=g_{\alpha\infty}\left(x^1\frac{v^0_\alpha}{A^0_\varphi(v^0_\alpha,v)},x^2-\frac{\sinh(\varphi)}{A^0_\varphi(v^0_\alpha,v)}x^1v^2,x^3-\frac{\sinh(\varphi)}{A^0_\varphi(v^0_\alpha,v)}x^1v^3,A^s_\varphi(v^0_\alpha,v)\right).
    \end{equation}
    \label{proposition_scattering_lineaire_f_phi}
\end{proposition}
\begin{proof}
     First recall 
    \begin{equation*}
        g_\alpha^\varphi(t,x,v)=f_\alpha^\varphi(t,x+t\widehat{v_\alpha},v).
    \end{equation*}
    We have 
    \begin{align*}
        g_\alpha^\varphi(tv^0_\alpha,x,v)&=g_\alpha\left(A_\varphi^0(tv^0_\alpha,x+tv),A_\varphi^s(tv^0_\alpha,x+tv)-A_\varphi^0(tv^0_\alpha,x+tv)\frac{A_\varphi^s(v^0_\alpha,v)}{A_\varphi^0(v^0_\alpha,v)},A_\varphi^s(v^0_\alpha,v)\right)\\
        &=g_\alpha\left(s\sqrt{m_\alpha^2+|w|^2},y,w\right),
    \end{align*}
    with $w:=A^s_\varphi(v^0_\alpha,v),\,s:=t+\frac{\sinh(\varphi)}{\sqrt{m_\alpha^2+|w|^2}}x^1,$ and
    \begin{equation*}
        y:=\left(x^1\left(\cosh(\varphi)-\frac{\sinh(\varphi)}{\sqrt{m_\alpha^2+|w|^2}}w^1\right),x^2-\frac{\sinh(\varphi)}{\sqrt{m_\alpha^2+|w|^2}}x^1w^2,x^3-\frac{\sinh(\varphi)}{\sqrt{m_\alpha^2+|w|^2}}x^1w^3\right).
    \end{equation*}
    For simplicity, let us write $w^0_\alpha:=\sqrt{m_\alpha^2+|w|^2}$. Since $w=A^s_\varphi(v^0_\alpha,v)$, we obtain $w^0_\alpha=A^0_\varphi(v^0_\alpha,v)$. Using the support of $g_\alpha^\varphi$ in \eqref{equation_support_g_alpha_varphi}, we know that there exists a constant $C_1$ such that
    \begin{equation*}
        g_\alpha^\varphi(tv^0_\alpha,x,v)=g_\alpha^\varphi(tv^0_\alpha,x,v)\1_{|x|\leq C_1(2+t)^{1/5}}=g_\alpha(sw^0_\alpha,y,w)\1_{|x|\leq C_1(2+t)^{1/5}}.
    \end{equation*}
    This implies,
    \begin{align}
        \int_{\R^3_x\times\R^3_v}|g_\alpha^\varphi(tv^0_\alpha,x,v)-g_{\alpha\infty}^\varphi(x,v)|\mathrm{d}x\mathrm{d}v\leq&\int_{\R^3_x\times\R^3_v} \left|g_\alpha(tw^0_\alpha,y,w)\1_{|x|\leq C_1(2+t)^{1/5}}-g_{\alpha\infty}^\varphi(x,v)\right|\mathrm{d}x\mathrm{d}v\label{equation_preuve_ecriture_Q_infini_1}\\
        & +\int_{\R^3_x\times\R^3_v} \left|g_\alpha(sw^0_\alpha,y,w)-g_\alpha(tw^0_\alpha,y,w)\right|\1_{|x|\leq C_1(2+t)^{1/5}}\mathrm{d}x\mathrm{d}v.\nonumber
    \end{align}
    Let us study the first term of the right-hand side. Recall that there exists $C_2>0$ such that $g_\alpha(tw^0_\alpha,y,w)$ vanishes when $|y|>C_2\log(2+t)$. Moreover, since $|w|\leq \beta$, there exists a constant $C(\varphi)$ such that $|x|\leq C(\varphi)|y|$. This implies 
    \begin{equation*}
        g_\alpha(tw^0_\alpha,y,w)=g_\alpha(tw^0_\alpha,y,w)\1_{|x|\leq C_2C(\varphi)\log(2+t)}.
    \end{equation*}
    Finally, for $t$ large enough, 
    \begin{equation*}
        \1_{|x|\leq C_2 C(\varphi)\log(2+t)}\1_{|x|\leq C_1(2+t)^{1/5}}=\1_{|x|\leq C_2C(\varphi)\log(2+t)},
    \end{equation*}
    and then 
    \begin{equation*}
        g_\alpha(tw^0_\alpha,y,w)\1_{|x|\leq C_1(2+t)^{1/5}}=g_\alpha(tw^0_\alpha,y,w).
    \end{equation*}
    It remains to show that the last integral on the right-hand side of \eqref{equation_preuve_ecriture_Q_infini_1} converges to $0$. By the mean value theorem and \eqref{equation_estimee_dt_g}, for $t$ large enough, we derive 
    \begin{align*}
        \int_{\R^3_x\times\R^3_v} \left|g_\alpha(sw^0_\alpha,y,w)-g_\alpha(tw^0_\alpha,y,w)\right|\1_{|x|\leq C_1(2+t)^{1/5}}\mathrm{d}x\mathrm{d}v\lesssim&\frac{1}{2+t}\int_{\R^3_x\times\R^3_v} |x|\1_{|x|\leq C_1(2+t)^{1/5}}\1_{|w|\leq \beta}\mathrm{d}x\mathrm{d}v\\
        \lesssim & (2+t)^{-1/5}.
    \end{align*}
    Finally, from \eqref{equation_preuve_ecriture_Q_infini_1}, we obtain
    \begin{align*}
        \int_{\R^3_x\times\R^3_v} |g_\alpha^\varphi(tv^0_\alpha,x,v)-g^\varphi_{\alpha\infty}(x,v)|\mathrm{d}x\mathrm{d}v \leq & \int_{\R^3_x\times\R^3_v} |g_\alpha(tw^0_\alpha,y,w)-g_{\alpha\infty}(y,w)|\mathrm{d}x\mathrm{d}v+ O\left((2+t)^{-1/5}\right).
    \end{align*}
    Then, consider the map
    \begin{equation*}
        G^\varphi:(x,v)\in\R^3_x\times\R^3_v\longmapsto(H^\varphi(x,v),D^\varphi(v))\in \R^3_x\times \R^3_v,
    \end{equation*}
    where 
    \begin{equation}
    \label{equation_definition_H_varphi}
    	H^\varphi(x,v):=\left(x^1\left(\cosh(\varphi)-\sinh(\varphi)\frac{w^1}{w^0_\alpha}\right),x^2-\sinh(\varphi)\frac{w^2}{w^0_\alpha}x^1,x^3-\sinh(\varphi)\frac{w^3}{w^0_\alpha}x^1\right)=y,
    \end{equation}
    and 
    \begin{equation*}
    	D^\varphi(v):=A_\varphi(v^0_\alpha,v)=w.
    \end{equation*}
    $G^\varphi$ is a $C^1$ diffeomorphism with Jacobian matrix given by
    \begin{equation*}
    	\nabla_{x,v}G^\varphi = \begin{pmatrix}
    		\nabla_x H^\varphi & \nabla_v H^\varphi\\
    		0 & \nabla_v D^\varphi
    	\end{pmatrix},
    \end{equation*}
    where $\nabla_x H^\varphi$, $\nabla_v H^\varphi$, and $\nabla_v D^\varphi$ denote the Jacobian matrices of $H^\varphi(\cdot,v)$, $H^\varphi(x,\cdot)$, and $D^\varphi$ respectively.  Since,
    \begin{equation}
    \label{equation_determinant_jac_H}
    	|\det \nabla_x H^\varphi(x,v)|=\left|\cosh(\varphi)-\sinh(\varphi)\frac{w^1}{w^0_\alpha}\right|=\frac{v^0_\alpha}{w^0_\alpha},\qquad |\det \nabla_v D^\varphi(v)|=\left|\cosh(\varphi)+\sinh(\varphi)\frac{v^1}{v^0_\alpha}\right|=\frac{w^0_\alpha}{v^0_\alpha},
    \end{equation}
    $G^\varphi$ has Jacobian $1$. We then apply the change of variables given by $G^\varphi$ to obtain 
    \begin{equation*}
        \int_{\R^3_x\times\R^3_v} |g_\alpha^\varphi(tv^0_\alpha,x,v)-g^\varphi_{\alpha\infty}(x,v)|\mathrm{d}x\mathrm{d}v \leq \int_{\R^3_y\times\R^3_w} |g_\alpha(tw^0_\alpha,y,w)-g_{\alpha\infty}(y,w)|\mathrm{d}y\mathrm{d}w+ O\left((2+t)^{-1/5}\right).
    \end{equation*}
    Finally, since $f_\alpha$ satisfies the linear scattering assumption \eqref{equation_fausse_linear_scattering1}, we derive the result.
\end{proof}
\begin{corollary}
    Let $A\in\mathrm{SO_0(3,1)}$ and $(f^A_\alpha,F^A)$ be defined as in \eqref{equation_premiere_definition_f_alpha_varphi}--\eqref{equation_premiere_definition_F_varphi}. $(f^A_\alpha,F^A)$ satisfies the linear scattering hypothesis \eqref{equation_fausse_linear_scattering1}.
    \label{corollaire_scattering_lineaire_f_A}
\end{corollary}
\begin{proof}
    This is a consequence of Propositions \ref{proposition_ecriture_boosts}, \ref{proposition_proprietes_conservees_rotations}, and \ref{proposition_scattering_lineaire_f_phi}.
\end{proof}
\begin{proposition}
    As expressed in \eqref{equation_corollaire_estimee_Q_infini}, let $Q_\infty^{\alpha,\varphi}(v_\alpha)$ be the limit of $\int_{\R^3_x}f_\alpha^\varphi(t,x,v)\mathrm{d}x$. The following expression holds for $Q_\infty^{\alpha,\varphi}$ 
    \begin{equation*}
        Q_\infty^{\alpha,\varphi}(v)=\frac{A_\varphi^0(v^0,v)}{v^0}Q^\alpha_\infty\left(A_\varphi^s(v^0,v)\right).
    \end{equation*}
    Let $Q_\infty^\varphi:=\sum_{1\leq \alpha\leq N}m_\alpha^3 e_\alpha Q^{\alpha,\varphi}_\infty$. Then  $Q_\infty^\varphi \neq 0$ and 
    \begin{equation*}
        Q_\infty^{\varphi}(v)=\frac{A_\varphi^0(v^0,v)}{v^0}Q_\infty\left(A_\varphi^s(v^0,v)\right).
    \end{equation*}
    \label{proposition_expression_Q_infini_phi}
\end{proposition}
\begin{proof}
    Recall Proposition \ref{proposition_scattering_lineaire_f_phi} and \eqref{equation_corollaire_estimee_Q_infini}. For $v\in\R^3_v$ fixed, consider the $C^1$ diffeomorphism $H^\varphi(\cdot,v)$ defined in \eqref{equation_definition_H_varphi}. By \eqref{equation_determinant_jac_H}, we know that $H^\varphi(\cdot,v)$ has Jacobian $\frac{v^0_\alpha}{A_\varphi^0(v^0_\alpha,v)}$. The result follows.
\end{proof}

\begin{remark}
    Note that this result can be shown without relying on the linear scattering hypothesis. In fact, we only use the expression of $g_\alpha^\varphi(tv^0_\alpha,x,v)$, equation \eqref{equation_corollaire_estimee_Q_infini} and the results presented in Remark \ref{remarque_consequence_hypothese1}.
\end{remark}

\begin{corollary}
    Let $Q_{\infty}^{\alpha,A}(v_\alpha)$ be the limit of $\int_{\R^3_x}f_\alpha^A(t,x,v)\mathrm{d}x$. The following expression holds for $Q_\infty^{\alpha,A}$
    \begin{equation*}
        Q_\infty^{\alpha,A}(v)=\frac{A^0(v^0,v)}{v^0}Q^\alpha_\infty\left(A^s(v^0,v)\right).
    \end{equation*}
    Let $Q_\infty^A:=\sum_{1\leq \alpha\leq N}m_\alpha^3 e_\alpha Q^{\alpha,A}_\infty$. Then  $Q_\infty^A \neq 0$ and 
    \begin{equation*}
        Q_\infty^{A}(v)=\frac{A^0(v^0,v)}{v^0}Q_\infty\left(A^s(v^0,v)\right).
    \end{equation*}
    \label{corollaire_expression_Q_infini_A}
\end{corollary}
\begin{proof}
    This follows from Propositions \ref{proposition_proprietes_conservees_rotations} and \ref{proposition_expression_Q_infini_phi}.
\end{proof}

%


\subsection{Proof of Theorem \ref{theoreme_pas_de_scattering_lineaire}}

Recall that we assumed $(f_\alpha,F)$ to be a solution to \eqref{equation_RVM1} that satisfies Hypothesis \ref{hypothese_suffisante} and the linear scattering assumption \eqref{equation_fausse_linear_scattering1}. Moreover, we assumed that $Q_\infty\neq 0$. In this context, we were able to compose $(f_\alpha,F)$ with $A\in\mathrm{SO}_0(3,1)$ and preserve Hypothesis \ref{hypothese_suffisante}, the linear scattering assumption \eqref{equation_fausse_linear_scattering1}, as well as $Q_\infty^A\neq 0$ (see Corollaries \ref{corollaire_hypothese_verifiee_f_A}, \ref{corollaire_scattering_lineaire_f_A}, and \ref{corollaire_expression_Q_infini_A}).\\
Moreover, we recall that by Propositions \ref{corollaire_Q_infini_L_non_nul} and \ref{proposition_Q_infini_nul_quand_L_non_nul}, it suffices to prove that $Q_\infty^A(0)\neq0$ for a certain $A\in\mathrm{SO}_0(3,1)$ to obtain a contradiction. Since $Q_\infty\neq0$, we know that there exists $v\in\R^3_v$ such that $Q_\infty(v)\neq0$. Moreover, according to Lemma \ref{lemme_existence_A_nul_en_espace}, there exists $A\in\mathrm{SO_0(3,1)}$ such that
\begin{equation*}
    A\left(1,0,0,0\right)=(v^0,v).
\end{equation*}
Finally, by Corollary \ref{corollaire_expression_Q_infini_A}, we have
\begin{equation*}
    Q_\infty^A(0)=v^0Q_\infty(v)\neq0.
\end{equation*}
This concludes the proof.



\printbibliography

\end{document}